\documentclass[reqno,12pt]{amsart}
\usepackage{amscd,amsfonts,amssymb}
\usepackage[T1,T2A]{fontenc}
\numberwithin{equation}{section}
\newtheorem{theorem}{Theorem}[section]
\newtheorem{lemma}{Lemma}[section]
\newtheorem{corollary}{Corollary}[section]

\theoremstyle{definition}
\newtheorem{definition}{Definition}[section]
\theoremstyle{remark}

\newtheorem{example}{Example}[section]

\author{V. V. Brovkin}
\address{Department of Differential Equations,
Faculty of Mechanics and Mathematics,
Mo\-s\-cow Lo\-mo\-no\-sov State University,
Vorobyovy Gory,
Moscow, 119992 Russia}
\email{brovvadim@rambler.ru}
\author{A. A. Kon'kov}
\address{Department of Differential Equations,
Faculty of Mechanics and Mathematics,
Mo\-s\-cow Lo\-mo\-no\-sov State University,
Vorobyovy Gory,
Moscow, 119992 Russia}
\email{konkov@mech.math.msu.su}
\title[On the existence of solutions]{
On the existence of solutions of the second boundary value problem for $p$-Laplacian on Riemannian manifolds
}
\thanks{}
\date{}

\begin{document}

\begin{abstract}
We obtain necessary and sufficient existence conditions for so\-lu\-ti\-ons of the boundary value problem
$$
	\Delta_p u = f
	\quad
	\mbox{on } M,
	\quad
	\left.
		\left|
			\nabla u
		\right|^{p - 2}
		\frac{\partial u}{\partial \nu}
	\right|_{
		\partial M
	}
	=
	h,
$$
where $p > 1$ is a real number, $M$ is a connected oriented complete Riemannian manifold with boundary,
and $\nu$ is the external normal vector to $\partial M$.
\end{abstract}

\maketitle

\section{Introduction}
\label{sec1}

Let $M$ be a connected oriented complete Riemannian manifold with boundary.
We consider the problem
\begin{equation}
	\Delta_p u = f
	\quad
	\mbox{on } M,
	\quad
	\left.
		\left|
			\nabla u
		\right|^{p - 2}
		\frac{\partial u}{\partial \nu}
	\right|_{
		\partial M
	}
	=
	h,
	\label{1.1}
\end{equation}
where
$
	\Delta_p u
	=
	\nabla_i
	(
		g^{ij}
		|\nabla u|^{p - 2}
		\nabla_j u
	),
$
$p > 1$,
is the $p$-Laplace operator,
$\nu$ is the external normal vector to $\partial M$,
and $f$ and $h$ are distributions from ${\mathcal D}' (M)$
with $\operatorname{supp} h \subset \partial M$.

As customary, by $g_{ij}$ we denote the metric tensor consistent with the Riemannian connection
and by $g^{ij}$ we denote the dual metric tensor, i.e.
$g_{ij} g^{jk} = \delta_i^k$. 
In so doing, $|\nabla u| = (g^{ij} \nabla_i u \nabla_j u)^{1/2}$.
Following~\cite{LU}, by $W_{p, loc}^1 (\omega)$, where $\omega$ is an open subset of $M$, 
we mean the space of measurable functions belonging to
$W_p^1 (\omega' \cap \omega)$ for any open set $\omega' \subset M$ with compact closure.
The space $L_{p, loc} (\omega)$ is defined analogously.

A function $u \in W_{p, loc}^1 (M)$ is called a solution of problem~\eqref{1.1} if
$$
	-
	\int_M
	g^{ij}
	|\nabla u|^{p - 2}
	\nabla_j u
	\nabla_i \varphi
	\,
	dV
	=
	(f - h, \varphi)
$$
for all $\varphi \in C_0^\infty (M)$, where $dV$ is the volume element of the manifold $M$.

As a condition at infinity, we require that solutions of~\eqref{1.1} satisfy the relation
\begin{equation}
	\int_M
	|\nabla u|^p
	\,
	dV
	<
	\infty.
	\label{1.2}
\end{equation}

Denote for brevity
\begin{equation}
	F = f - h.
	\label{1.3}
\end{equation}
In the partial case of $f \in L_{1, loc} (M)$ and $h \in L_{1, loc} (\partial M)$, we obviously have
$$
	(F, \varphi)
	=
	\int_M
	f
	\varphi
	\,
	dV
	-
	\int_{\partial M}
	h
	\varphi
	\,
	dS
$$
for all $\varphi \in C_0^\infty (M)$, 
where $dV$ is the volume element of $M$ and $dS$ is the volume element of $\partial M$.

\begin{definition}\label{d1.1}
The capacity of a compact set $K \subset \omega$ relative to an open set $\omega \subset M$ is defined by
$$
    \operatorname{cap}_p (K, \omega)
    =
    \inf_\varphi \int_\omega
    |\nabla \varphi|^p
    \,
    dx,
$$
where the infimum is taken over all functions
$\varphi \in C_0^\infty (\omega)$ that are identically equal to one in a neighborhood of $K$.
In the case of $\omega = M$, we write $\operatorname{cap}_p (K)$ instead of $\operatorname{cap}_p (K, M)$.
For an arbitrary closed set $H \subset M$, we put
$$
	\operatorname{cap}_p (H)
	=
	\sup_K
	\operatorname{cap}_p (K),
$$
where the supremum is taken over all compact sets $K \subset H$.
The capacity of the empty set is assumed to be equal to zero.
\end{definition}

In the case of $M = {\mathbb R}^n$, $n \ge 3$,
the capacity $\operatorname{cap}_2 (K)$ coincides with the well-known Wiener capacity~\cite{Landkoff}.
It can be easily shown that $\operatorname{cap}_p (K, \omega)$ has the following natural properties.

\begin{enumerate}
\item[(a)]{\it Monotonicity.}
If
$K_1 \subset K_2$ 
and
$\omega_2 \subset \omega_1$,
then
$$
	\operatorname{cap}_p (K_1,\omega_1) \le \operatorname{cap}_p (K_2,\omega_2).
$$
\item[(b)]{\it  Semi-additivity.}
If $K_1$ and $K_2$ are compact subsets of an open set $\omega$, then
$$
	\operatorname{cap}_p (K_1 \cup K_2, \omega) 
	\le
	\operatorname{cap}_p (K_1,\omega)
	+
	\operatorname{cap}_p (K_2,\omega).
$$
\end{enumerate}

\begin{definition}\label{d1.2}
Manifold $M$ is called $p$-hyperbolic, if its capacity is positive,
i.e. $\operatorname{cap}_p (M) > 0$;
otherwise this manifold is called $p$-parabolic.
\end{definition}

If $M$ is a compact manifold, it is obviously $p$-parabolic.
It can be also shown that ${\mathbb R}^n$ is a $p$-parabolic manifold for $p \ge n$
and a $p$-hyperbolic manifold for $p < n$. 

By $L_p^1 (\omega)$, where $\omega$ is an open subset of $M$,
we denote the space of distributions $u \in {\mathcal D}' (\omega)$
for which $\nabla u \in L_p (\omega)$. 
The semi norm in $L_p^1 (\omega)$ is defined as 
$$
	\| u \|_{
		L_p^1 (\omega)
	}
	=
	\left(
		\int_\omega
		|\nabla u|^p
		\,
		dV
	\right)^{1 / p}.
$$
It is known~\cite{Mazya} that ${L_p^1 (\omega) \subset L_p (K)}$ for any compact set $K \subset \omega$. 
It can be also shown that $L_p^1 (\omega) / {{<}1{>}}$ 
is a uniformly convex and therefore reflexive Banach space.
By $\stackrel{\rm \scriptscriptstyle o}{L}\!\!{}_p^1 (\omega)$ 
we denote the closure of $C_0^\infty (\omega)$ in $L_p^1 (\omega)$.
By $\stackrel{\rm \scriptscriptstyle o}{L}\!\!{}_p^1 (\omega)^*$  
we mean the dual space to
$\stackrel{\rm \scriptscriptstyle o}{L}\!\!{}_p^1 (\omega)$ 
or, in other words, the space of linear continuous functionals on
$\stackrel{\rm \scriptscriptstyle o}{L}\!\!{}_p^1 (\omega)$.
The norm of a functional $l \in {\stackrel{\rm \scriptscriptstyle o}{L}\!\!{}_p^1 (\omega)^*}$ is defined as
$$
	\| l \|_{
		\stackrel{\rm \scriptscriptstyle o}{L}{}_p^1 (\omega)^*
	}
	=
	\sup_{
		\varphi \in C_0^\infty (\omega),
		\;
		\| \varphi \|_{
			L_p^1 (\omega)
		}
		=
		1
	}
	|(l, \varphi)|.
$$

For the solvability of problem~\eqref{1.1}, \eqref{1.2} it is necessary and sufficient that
the functional $F$ defined by~\eqref{1.3} is continuous in the space
${\stackrel{\rm \scriptscriptstyle o}{L}\!\!{}_p^1 (M)}$.
Indeed, if $u$ is a solution of~\eqref{1.1}, \eqref{1.2}, then
\begin{equation}
	-
	\int_M
	g^{ij}
	|\nabla u|^{p - 2}
	\nabla_j u
	\nabla_i \varphi
	\,
	dV
	=
	(F, \varphi)
	\label{1.4}
\end{equation}
for all $\varphi \in C_0^\infty (M)$, whence in accordance with the H\"older inequality we obtain
$$
	|(F, \varphi)|
	\le
	\| u \|_{
		L_p^1 (M)
	}^{p - 1}
	\| \varphi \|_{
		L_p^1 (M)
	}.
$$
Defining $F$ by continuity to the whole space 
$\stackrel{\rm \scriptscriptstyle o}{L}\!\!{}_p^1 (M)$,
we complete the proof of necessity. 
To prove sufficiency, let us take a sequence 
$\varphi_i \in C_0^\infty (M)$, $i = 1,2,\ldots$,
such that
$$
	\lim_{i \to \infty} 
	J (\varphi_i)
	=
	\inf_{
		\varphi \in C_0^\infty (M)
	}
	J (\varphi),
$$
where
$$
	J (\varphi)
	=
	\frac{1}{p}
	\int_M
	| \nabla \varphi |^p
	\,
	dV
	+
	(F, \varphi).
$$
Since $F$ is a continuous functional in
$\stackrel{\rm \scriptscriptstyle o}{L}\!\!{}_p^1 (M)$,
the sequence $\{ \varphi_i \}_{i=1}^\infty$ is bounded 
in the semi norm of the space $L_p^1 (M)$, in particular,
$$
	\lim_{i \to \infty} 
	J (\varphi_i)
	=
	\inf_{
		\varphi \in C_0^\infty (M)
	}
	J (\varphi)
	>
	- \infty.
$$
We select from the sequence $\varphi_i + {{<}1{>}} \in {L_p^1 (M) / {{<}1{>}}}$, $i = 1,2,\ldots$,
a subsequence
$\varphi_{i_j} + {{<}1{>}}$, $j = 1,2,\ldots$,
converging weakly to some element $u + {{<}1{>}}$ of the space $L_p^1 (M) / {{<}1{>}}$.
In view of the reflexivity of $L_p^1 (M) / {{<}1{>}}$, such a sequence exists.
We denote by $R_m$ the convex hull of the set $\{ \varphi_{i_j} \}_{j \ge m}$.
By Mazur's theorem,
there is a sequence $r_m \in R_m$, $m = 1,2,\ldots$, such that
$$
	\| u - r_m \|_{
		L_p^1 (M)
	}
	\to
	0
	\quad
	\mbox{as } m \to \infty.
$$
Since $r_m \in C_0^\infty (M)$, $m = 1,2,\ldots$, this implies the inclusion
$u \in {\stackrel{\rm \scriptscriptstyle o}{L}\!\!{}_p^1 (M)}$.
We can also assert that
$$
	J (r_m)
	\le
	\sup_{j \ge m}
	J (\varphi_{i_j}),
	\quad
	m = 1,2,\ldots,
$$
as $J$ is a convex functional. 
Passing to the limit as $m \to \infty$ in the last inequality, we obtain
$$
	J (u)
	\le
	\inf_{
		\varphi \in C_0^\infty (M)
	}
	J (\varphi).
$$
Since the converse inequality is obvious, this yields
$$
	J (u)
	=
	\inf_{
		\varphi \in C_0^\infty (M)
	}
	J (\varphi),
$$
whence~\eqref{1.4} follows according to the variational principle.
Thus, $u$ is a solution of problem~\eqref{1.1}, \eqref{1.2}.

Exterior boundary value problems traditionally attract the attention of ma\-the\-ma\-ti\-ci\-ans~[1--7].
In the paper presented to your attention, we give necessary and sufficient conditions 
for the solvability of the Neumann problem on Riemannian manifolds.
These conditions are different for $p$-hyperbolic and $p$-parabolic mani\-folds.
For example, in the simple case that $M = {\mathbb R}^n \setminus B_1$, 
where $B_1$ is a unit ball in ${\mathbb R}^n$, $n \ge 2$, the exterior Neumann problem
$$
	\Delta_p u = 0
	\quad
	\mbox{on } {\mathbb R}^n \setminus B_1,
	\quad
	\left.
		\left|
			\nabla u
		\right|^{p - 2}
		\frac{\partial u}{\partial \nu}
	\right|_{
		\partial B_1
	}
	=
	h,
	\quad
	\int_{{\mathbb R}^n \setminus B_1}
	| \nabla u|^p
	\,
	dx
	<
	\infty
$$
has a solution for any function $h \in L_{p / (p - 1)} (\partial B_1)$ if $n > p$.
On the other hand, in the case of $n \le p$, for a solution to exist it is necessary and sufficient that
$$
	\int_{\partial B_1}
	h
	\,
	dS
	=
	0.
$$
In this sense, $p$-parabolic manifolds are similar to bounded domains in ${\mathbb R}^n$ 
(see Corollaries~\ref{c2.1} and~\ref{c2.2}).

\section{The case of the functional $F$ with a compact support}\label{PartialCase}

\begin{theorem}\label{t2.1}
Let $M$ be a $p$-hyperbolic manifold and the functional $F$ defined by~\eqref{1.3} have a compact support.
Then for problem~\eqref {1.1}, \eqref{1.2} to have a solution, it is necessary and sufficient that 
$F$ is a continuous functional in $\stackrel{\rm \scriptscriptstyle o}{L}\!\!{}_p^1 (\omega)$ 
for some open set $\omega$ such that $\operatorname{supp} F \subset \omega$.
\end{theorem}

\begin{theorem}\label{t2.2}
Let $M$ be a $p$-parabolic manifold and the functional $F$ defined by~\eqref{1.3} have a compact support.
Then for problem~\eqref {1.1}, \eqref{1.2} to have a solution, it is necessary and sufficient that 
$F$ is a continuous functional in $\stackrel{\rm \scriptscriptstyle o}{L}\!\!{}_p^1 (\omega)$ 
for some open set $\omega$ such that $\operatorname{supp} F \subset \omega$ and, moreover,
\begin{equation}
	\lim_{s \to \infty}
	(F, \eta_s)
	=
	0
	\label{t2.2.1}
\end{equation}
for some sequence $\eta_s \in C_0^\infty (M)$ such that
\begin{equation}
	\lim_{s \to \infty}
	\| \eta_s \|_{
		L_p^1 (M)
	}
	=
	0
	\quad
	\mbox{and}
	\quad
	\left.
		\eta_s
	\right|_{
		K
	}
	=
	1,
	\quad
	s = 1,2,\ldots,
	\label{t2.2.2}
\end{equation}
where $K$ is a compact set of positive measure.
\end{theorem}

\begin{corollary}\label{c2.1}
Let $M$ be a $p$-hyperbolic manifold with compact boundary and
$h$ be a functional from ${\mathcal D}' (M)$ such that $\operatorname{supp} h \subset \partial M$.
Then the problem
\begin{equation}
	\Delta_p u = 0
	\quad
	\mbox{on } M,
	\quad
	\left.
		\left|
			\nabla u
		\right|^{p - 2}
		\frac{\partial u}{\partial \nu}
	\right|_{
		\partial M
	}
	=
	h,
	\label{c2.1.1}
\end{equation}
has a solution satisfying condition~\eqref{1.2} if and only if $h$ is a continuous functional in
$\stackrel{\rm \scriptscriptstyle o}{L}\!\!{}_p^1 (\omega)$ 
for some open set $\omega$ such that $\partial M \subset \omega$.
\end{corollary}

\begin{corollary}\label{c2.2}
Let $M$ be a $p$-parabolic manifold with compact boundary and
$h$ be a functional from ${\mathcal D}' (M)$ such that $\operatorname{supp} h \subset \partial M$.
Then for problem~\eqref{c2.1.1}, \eqref{1.2} to have a solution,
it is necessary and sufficient that $h$ is a continuous functional in
$\stackrel{\rm \scriptscriptstyle o}{L}\!\!{}_p^1 (\omega)$ 
for some open set $\omega$ such that $\partial M \subset \omega$ and, moreover,
\begin{equation}
	(h, 1) = 0.
	\label{c2.2.1}
\end{equation}
\end{corollary}

Corollaries~\ref{c2.1} and~\ref{c2.2} immediately follow from Theorems~\ref{t2.1} and~\ref{t2.2}.
Assuming without loss of generality that $\overline{\omega}$ is a compact set, 
we only note that condition~\eqref{c2.2.1} implies~\eqref{t2.2.1} 
for any sequence $\eta_s \in C_0^\infty (M)$
satisfying~\eqref{t2.2.2}, where $K = \overline{\omega}$.
In the case of $p$-parabolic manifold, such a sequence obviously exists.
On the other hand, if $h \in {\stackrel{\rm \scriptscriptstyle o}{L}\!\!{}_p^1 (M)^*}$,
then~\eqref{t2.2.1} is valid for any sequence
$\eta_s \in C_0^\infty (M)$ satisfying~\eqref{t2.2.2}, where $K = \overline{\omega}$.
This in turn implies the validity of~\eqref{c2.2.1}.
 
To prove Theorems~\ref{t2.1} and~\ref{t2.2}, we need the following lemmas.

\begin{lemma}\label{l2.1}
Let $\operatorname{cap} (K) = 0$ for some compact set $K$ of positive measure.
Then $M$ is a $p$-parabolic manifold.
\end{lemma}

\begin{lemma}\label{l2.2}
Let $M$ be a $p$-hyperbolic manifold. Then for any compact set $K$ the space
$\stackrel{\rm \scriptscriptstyle o}{L}\!\!{}_p^1 (M)$ 
is continuously embedded in $L_p (K)$. In other words,
\begin{equation}
	\| \varphi \|_{
		L_p (K)
	}
	\le
	C
	\| \varphi \|_{
		L_p^1 (M)
	}
	\label{l2.2.1}
\end{equation}
for all $\varphi \in C_0^\infty (M)$, where the constant $C > 0$ does not depend of $\varphi$.
\end{lemma}

In the case of $p = 2$, Lemmas~\ref{l2.1} and~\ref{l2.2} are proved in~\cite[Chapter~3, \S 1]{KonkovPhD}.
For $p > 1$, they are proved in a similar way.
For the convenience of readers, we give this proof in full.

\begin{proof}[Proof of Lemma~$\ref{l2.1}$]
Let $H$ be a compact set containing $K$ and $\omega$ be a domain with compact closure such that
$H \subset \omega$.
We represent $M$ as a union of Lipschitz domains $\omega_i$ with compact closures such that 
$\overline{\omega} \subset \omega_i \subset \omega_{i+1}$, $i = 1,2,\ldots$.
Let us denote by $u_i$ the solution of the problem
$$
	\Delta_p u_i = 0
	\quad
	\mbox{on } \omega_i \setminus K,
	\quad
	\left.
		u_i
	\right|_K
	=
	1,
	\quad
	\left.
		u_i
	\right|_{\partial \omega_i}
	=
	0,
	\quad
	\left.
		\left|
			\nabla u
		\right|^{p - 2}
		\frac{\partial u}{\partial \nu}
	\right|_{
		\omega_i \cap \partial M
	}
	=
	0.
$$
By the maximum principle, we have $0 \le u_i (x) \le u_{i+1} (x) \le 1$ for all $x \in \omega_i$, $i = 1,2,\ldots$.
At the same time,
$$
	\| u_i \|_{
		L_p^1 (\omega_i)
	}
	=
	\operatorname{cap}_p (K, \omega_i)
	\to
	0
	\quad
	\mbox{as } i \to \infty.
$$
Hence, $\{ u_i \}_{i=1}^\infty$ is a fundamental sequence in $W_p^1 (\omega)$.
Since $\operatorname{mes} K > 0$, we obviously obtain $u_i \to 1$ in $W_p^1 (\omega)$ as $i \to \infty$.
It is known~\cite{LU} that the sequence $\{ u_i \}_{i=1}^\infty$ is bounded in a H\"older norm 
in some neighborhood of the set $\partial \omega$; therefore, it has a subsequence
converging to one uniformly on $\partial \omega$.
By the maximum principle, on the set $\omega$ the function $u_i$ does not exceed one and is not less
than the exact lower bound of this function on $\partial \omega$.
Thus, $u_i \to 1$ uniformly on $\omega$ as $i \to \infty$.
Further, let us take a function $\eta \in C^\infty ({\mathbb R})$
which is equal to zero on $(- \infty, 1 / 4]$ and is equal to one on $[3 / 4, \infty)$.
It is easy to see that $\eta \circ u_i \in \stackrel{\rm \scriptscriptstyle o}{W}\!\!{}_p^1 (\omega_i)$ 
and, moreover, $\eta \circ u_i = 1$ on $\omega$ for all enough large~$i$. 
Since
$$
	\|\eta \circ u_i \|_{
		L_p^1 (\omega_i)
	}
	\le
	\|\eta' \|_{
		C ({\mathbb R})
	}
	\|u_i \|_{
		L_p^1 (\omega_i)
	}
	\to
	0
	\quad
	\mbox{as } i \to \infty,
$$
this allows us to assert that
$$
	\operatorname{cap}_p (H)
	\le
	\lim_{i \to \infty}
	\|\eta \circ u_i \|_{
		L_p^1 (\omega_i)
	}
	=
	0.
$$
\end{proof}

\begin{proof}[Proof of Lemma~$\ref{l2.2}$]
We take a Lipschitz domain $\omega$ with compact closure such that $K \subset \omega$.
Assume to the contrary that there is a sequence of functions $\varphi_i \in C_0^\infty (M)$,
satisfying the conditions
\begin{equation}
	\lim_{i \to \infty}
	\| \varphi_i \|_{
		L_p^1 (M)
	}
	=
	0
	\quad
	\mbox{and}
	\quad
	\| \varphi_i \|_{
		L_p (\omega)
	}
	=
	\operatorname{mes} \omega > 0,
	\quad
	i = 1,2,\ldots.
	\label{pl2.2.1}
\end{equation}
Since $W_p^1 (\omega)$ is completely continuously embedded in $L_p (\omega)$, there is a subsequence
of the sequence $\{ \varphi_i \}_{j = 1}^\infty$ converging in $L_p (\omega)$.
In order not to clutter up indices, we denote this subsequence also by
$\{ \varphi_i \}_{i = 1}^\infty$.
In view of~\eqref{pl2.2.1}, we have $\varphi_i \to 1$ in $W_p^1 (\omega)$ as $i \to \infty$;
therefore, some subsequence of this sequence which we again denote by $\{ \varphi_i \}_{i = 1}^\infty$ 
converges to one almost everywhere on $\omega$.
According to Egorov's theorem, there is a set $E \subset \omega$ of positive measure such that
$\{ \varphi_i \}_{i = 1}^\infty$ tends to one uniformly on $E$.
Since $\varphi_i$ are continuous functions, the sequence $\{ \varphi_i \}_{i = 1}^\infty$ 
also tends to one uniformly on the compact set $\overline{E}$.
Thus, taking the function $\eta$ the same as in the proof of the Lemma~\ref{l2.1}, we have
$$
	\operatorname{cap}_p (\overline{E})
	\le
	\lim_{i \to \infty}
	\|\eta \circ \varphi_i \|_{
		L_p^1 (M)
	}
	=
	0,
$$
whence in accordance with Lemma~\ref{l2.1} it follows that $M$ is a $p$-parabolic manifold.
This contradiction proves the inequality
$$
	\| \varphi \|_{
		L_p (\omega)
	}
	\le
	C
	\| \varphi \|_{
		L_p^1 (M)
	}
$$
for all $\varphi \in C_0^\infty (M)$, where the constants $C > 0$ does not depend of $\varphi$,
from which~\eqref{l2.2.1} follows at once.
\end{proof}

\begin{proof}[Proof of Theorem~$\ref{t2.1}$]
The necessity is obvious. 
Indeed, from the continuity of the functional $F$ in the space
$\stackrel{\rm \scriptscriptstyle o}{L}\!\!{}_p^1 (M)$, 
it follows that $F$ is also continuous in
$\stackrel{\rm \scriptscriptstyle o}{L}\!\!{}_p^1 (\omega)$ 
for any open subset $\omega$ of the manifolds $M$.
We prove the sufficiency.
Assume that $F \in {\stackrel{\rm \scriptscriptstyle o}{L}\!\!{}_p^1 (\omega)^*}$ 
for some open set $\omega$ with
$\operatorname{supp} F \subset \omega$.
Let us show that $F \in {\stackrel{\rm \scriptscriptstyle o}{L}\!\!{}_p^1 (M)^*}$.
Take a function $\psi \in C_0^\infty (\omega)$ equal to one in a neighborhood of $\operatorname{supp} F$.
It is easy to see that
$$
	|(F, \psi \varphi)|
	\le
	\| F \|_{
		\stackrel{\rm \scriptscriptstyle o}{L}{}_p^1 (\omega)^*
	}
	\| \psi \varphi \|_{
		L_p^1 (\omega)
	}
$$
for all $\varphi \in C_0^\infty (M)$, whence in accordance with the fact that
$(F, \varphi) = (F, \psi \varphi)$
and
$$
	\| \psi \varphi \|_{
		L_p^1 (\omega)
	}
	\le
	\| \psi\|_{
		C (\omega)
	}
	\|\varphi \|_{
		L_p^1 (\omega)
	}
	+
	\|\nabla \psi\|_{
		C (\omega)
	}
	\|\varphi \|_{
		L_p (\operatorname{supp}  \psi)
	}
$$
we obtain
$$
	|(F, \varphi)|
	\le
	\| F \|_{
		\stackrel{\rm \scriptscriptstyle o}{L}{}_p^1 (\omega)^*
	}
	\left(
		\| \psi\|_{
			C (\omega)
		}
		\|\varphi \|_{
			L_p^1 (\omega)
		}
		+
		\|\nabla \psi\|_{
			C (\omega)
		}
		\|\varphi \|_{
			L_p (\operatorname{supp} \psi)
		}
	\right)
$$
for all $\varphi \in C_0^\infty (M)$. 
At the same time, Lemma~\ref{l2.2} implies that
$$
	\|\varphi \|_{
		L_p (\operatorname{supp}  \psi)
	}
	\le
	C
	\| \varphi \|_{
		L_p^1 (M)
	}
$$
for all $\varphi \in C_0^\infty (M)$, where the constant $C > 0$ does not depend of $\varphi$.

Thus, to complete the proof, it remains to combine the last two estimates.
\end{proof}

\begin{proof}[Proof of Theorem~$\ref{t2.2}$]
As in the case of Theorem~\ref{t2.1}, we need only to prove the sufficiency as the necessity is obvious.
Let $\Omega$ be a Lipschitz domain with compact closure containing $K$ and $\operatorname{supp} F$.
Without loss of generality, it can be assumed that the norms
$
	\| \eta_s \|_{
		W_p^1 (\Omega)
	}
$
are bounded by a constant independent of $s$.
If the last condition is not valid, we replace $\eta_s$ with
\begin{equation}
	\tilde \eta_s (x)
	=
	\left\{
		\begin{aligned}
			&
			0,
			&
			&
			\eta_s (x) \le 0,
			\\
			&
			\eta_s (x),
			&
			&
			0 < \eta_s (x) < 1,
			\\
			&
			1,
			&
			&
			1 \le \eta_s (x),
		\end{aligned}
	\right.
	\quad
	s = 1,2,\dots.
	\label{pt2.2.1}
\end{equation}
Since $W_p^1 (\Omega)$ is completely continuous embedded in $L_p (\Omega)$,
there exists a sub\-sequence of the sequence $\{ \eta_s \}_{s = 1}^\infty$
converging in ${L_p (\Omega)}$. 
Denote this subsequence also by $\{ \eta_s \}_{s = 1}^\infty$.
In view of~\eqref{t2.2.2}, we obtain
\begin{equation}
	\| 1 - \eta_s \|_{
		W_p^1 (\Omega)
	}
	\to
	0
	\quad
	\mbox{as } s \to \infty.
	\label{pt2.2.2}
\end{equation}

Assume that $\varphi \in {C_0^\infty (M)}$.
By the Poincare inequality,
\begin{equation}
	\int_\Omega
	|\varphi - \alpha|^p
	\,
	dV
	\le
	C
	\int_\Omega
	|\nabla \varphi|^p
	\,
	dV,
	\label{pt2.2.3}
\end{equation}
where
\begin{equation}
	\alpha
	=
	\frac{
		1
	}{
		\operatorname{mes} \Omega
	}
	\int_\Omega
	\varphi
	\,
	dV.
	\label{pt2.2.4}
\end{equation}
Hereinafter in the proof of Theorem~\ref{t2.2}, 
by $C$ we mean various positive constants depending only on 
$p$, $\omega$, $\Omega$, and the support of the functional $F$.
Take a function $\psi \in C_0^\infty (\omega \cap \Omega)$ equal to one 
in a neighborhood of $\operatorname{supp} F$. 
Let us denote
\begin{equation}
	\varphi_j' 
	= 
	(\varphi - \alpha \eta_j)
	(1 - \psi)
	\quad
	\mbox{and}
	\quad
	\varphi_j''
	=
	(\varphi - \alpha \eta_j)
	\psi,
	\quad
	j = 1,2,\ldots.
	\label{pt2.2.5}
\end{equation}
We have
$\varphi = \varphi_j' + \varphi_j'' + \alpha \eta_j$;
therefore,
\begin{equation}
	|(F, \varphi)|
	\le
	|(F, \varphi_j')|
	+
	|(F, \varphi_j'')|
	+
	|\alpha| |(F, \eta_j)|,
	\quad
	j = 1,2,\ldots,
	\label{pt2.2.6}
\end{equation}
Since $(F, \varphi_j') = 0$ and $\varphi_j'' \in C_0^\infty (\omega)$, this obviously implies the estimate
$$
	|(F, \varphi)|
	\le
	\| F \|_{
		\stackrel{\rm \scriptscriptstyle o}{L}{}_p^1 (\omega)^*
	} 
	\| \varphi_j'' \|_{
		L_p^1 (\omega)
	}
	+
	|\alpha| |(F, \eta_j)|,
	\quad
	j = 1,2,\ldots.
$$
Combining it with the inequality
$$
	\| \varphi_j'' \|_{
		L_p^1 (\omega)
	}
	\le
	\| \psi \|_{
		C (\Omega)
	}
	\| \varphi - \alpha \eta_j \|_{
		L_p^1 (\Omega)
	}
	+
	\|\nabla \psi \|_{
		C (\Omega)
	}
	\| \varphi - \alpha \eta_j \|_{
		L_p (\Omega)
	},
$$
we obtain
\begin{align*}
	|(F, \varphi)|
	\le
	{}
	&
	C
	\| F \|_{
		\stackrel{\rm \scriptscriptstyle o}{L}{}_p^1 (\omega)^*
	} 
	\left(
		\| \varphi - \alpha \eta_j \|_{
			L_p^1 (\Omega)
		}
		+
		\| \varphi - \alpha \eta_j \|_{
			L_p (\Omega)
		}
	\right)
	\\
	&
	{}
	+
	|\alpha| |(F, \eta_j)|,
	\quad
	j = 1,2,\ldots.
\end{align*}
Passing to the limit in the last formula as $j \to \infty$ with account of the relations
\begin{equation}
	\| \varphi - \alpha \eta_j \|_{
		L_p^1 (M)
	}
	\le
	\| \varphi \|_{
		L_p^1 (M)
	}
	+
	|\alpha|
	\| \eta_j \|_{
		L_p^1 (M)
	}
	\to
	\| \varphi \|_{
		L_p^1 (M)
	}
	\quad
	\mbox{as } j \to \infty
	\label{pt2.2.7}
\end{equation}
and
\begin{align}
	&
	\| \varphi - \alpha \eta_j \|_{
		L_p (\Omega)
	}
	=
	\| \varphi - \alpha + \alpha (1 - \eta_j) \|_{
		L_p (\Omega)
	}
	\le
	\| \varphi - \alpha \|_{
		L_p (\Omega)
	}
	+
	|\alpha|
	\| 1 - \eta_j \|_{
		L_p (\Omega)
	}
	\nonumber
	\\
	&
	\qquad
	{}
	\le
	C
	\| \varphi \|_{
		L_p^1 (\Omega)
	}
	+
	|\alpha|
	\| 1 - \eta_j \|_{
		L_p (\Omega)
	}
	\to
	C
	\| \varphi \|_{
		L_p^1 (\Omega)
	}
	\quad
	\mbox{as } j \to \infty,
	\label{pt2.2.8}
\end{align}
we have
$$
	|(F, \varphi)|
	\le
	C
	\| F \|_{
		\stackrel{\rm \scriptscriptstyle o}{L}{}_p^1 (\omega)^*
	} 
	\| \varphi \|_{
		L_p^1 (M)
	}.
$$
The proof is completed.
\end{proof}

\section{The case of the general functional $F$}\label{GeneralCase}

We assume that the manifold $M$ admits a locally finite cover
\begin{equation}
	M
	=
	\bigcup_{i = 1}^\infty
	\Omega_i
	\label{3.1}
\end{equation}
of the multiplicity $k < \infty$,
where $\Omega_i$ are Lipschitz domains with compact closure such that
$\Omega_i \cap \Omega_{i + 1} \ne \emptyset$, $i = 1,2,\dots$. 
In so doing, let $\gamma : M \to (0, \infty)$ be a measurable function
separated from zero and infinity on every compact subset of the manifold $M$ and
$\psi_i \in C_0^\infty (\Omega_i)$ be a partition of unity on $M$ such that
\begin{equation}
	|\nabla \psi_i (x)|^p
	\le
	\gamma (x),
	\quad
	x \in \Omega_i,
	\quad
	i = 1,2,\dots.
	\label{3.5}
\end{equation}

We need the following well-known assertion.

\begin{lemma}[Poincare inequality]\label{l3.1}
Let $\omega$ be a Lipschitz domain with compact closure.
Then
\begin{equation}
	\int_\omega
	\gamma 
	|u - \bar u|
	\,
	dV
	\le
	C
	\int_\omega
	\gamma^{1 - 1 / p}
	|\nabla u|
	\,
	dV
	\label{l3.1.1}
\end{equation}
for any function $u \in W_1^1 (\omega)$, where 
$$
	\bar u
	=
	\frac{
		\int_\omega
		u
		\,
		dV
	}{
		\int_\omega
		\gamma 
		\,
		dV
	}
$$
and the constant $C > 0$ does not depend on $u$.
\end{lemma}

We also assume that
\begin{equation}
	\sup_{
		i \in {\mathbb N}
	}
	C_i
	\left(
		\frac{
			1
		}{
			\int_{\Omega_i}
			\gamma
			\,
			dV
		}
		+
		\frac{
			1
		}{
			\int_{\Omega_{i + 1}}
			\gamma
			\,
			dV
		}
	\right)
	\sum_{j = 1}^i
	\int_{\Omega_j}
	\gamma
	\,
	dV
	<
	\infty
	\label{3.2}
\end{equation}
in the case of the $p$-hyperbolic manifold $M$ and
\begin{equation}
	\sup_{
		i \in {\mathbb N}
	}
	C_i
	\left(
		\frac{
			1
		}{
			\int_{\Omega_i}
			\gamma
			\,
			dV
		}
		+
		\frac{
			1
		}{
			\int_{\Omega_{i + 1}}
			\gamma
			\,
			dV
		}
	\right)
	\sum_{j = i + 1}^\infty
	\int_{\Omega_j}
	\gamma
	\,
	dV
	<
	\infty
	\label{3.3}
\end{equation}
in the case of the $p$-parabolic manifold $M$,
where $\mathbb N$ is the set of positive integers and
$C_i > 0$ is the constant in~\eqref{l3.1.1} for 
$\omega = \Omega_i \cup \Omega_{i + 1}$.

\begin{theorem}\label{t3.1}
Let $M$ be a $p$-hyperbolic manifold.
Then for problem~\eqref {1.1}, \eqref{1.2} to have a solution, it is necessary and sufficient that 
\begin{equation}
	\sum_{i = 1}^\infty
	\| F \|_{
		\stackrel{\rm \scriptscriptstyle o}{L}{}_p^1 (\Omega_i)^*
	}^{p / (p - 1)}
	<
	\infty,
	\label{t3.1.1}
\end{equation}
where $F$ is defined by~\eqref{1.3}. 
\end{theorem}

\begin{theorem}\label{t3.2}
Let $M$ be a  $p$-parabolic manifold.
Then for problem~\eqref {1.1}, \eqref{1.2} to have a solution, it is necessary and sufficient 
that~\eqref{t3.1.1} holds and, moreover, conditions~\eqref{t2.2.1} and~\eqref{t2.2.2} are valid
for some sequence $\eta_s \in C_0^\infty (M)$.
\end{theorem}

The proof of Theorems~\ref{t3.1} and~\ref{t3.2} relies on Lemmas~\ref{l3.2} -- \ref{l3.4}.

\begin{lemma}\label{l3.2}
Let $\omega_1$ and $\omega_2$ are measurable subsets of a Lipschitz domain $\omega \Subset M$ such that
$$
	\gamma_i 
	= 
	\int_{\omega_i}
	\gamma 
	\,
	dV
	>
	0,
	\quad
	i = 1,2.
$$
Then
$$
	\frac{
		1
	}{
		\gamma_1
	}
	\int_{\omega_1}
	\gamma 
	|u|
	\,
	dV
	\le
	C
	\left(
		\frac{
			1
		}{
			\gamma_1
		}
		+
		\frac{
			1
		}{
			\gamma_2
		}
	\right)
	\int_\omega
	\gamma^{1 - 1 / p}
	|\nabla u|
	\,
	dV
	+
	\frac{
		1
	}{
		\gamma_2
	}
	\int_{\omega_2}
	\gamma
	|u|
	\,
	dV
$$
for any function $u \in W_1^1 (\omega)$, where $C > 0$ is the constant in~\eqref{l3.1.1}.
\end{lemma}

\begin{proof}
Taking into account~\eqref{l3.1.1}, we have
$$
	\int_{\omega_1}
	\gamma 
	|u - \bar u|
	\,
	dV
	\le
	C
	\int_\omega
	\gamma^{1 - 1 / p}
	|\nabla u|
	\,
	dV.
$$
By the inequality $|u| - |\bar u| \le |u - \bar u|$, this implies that
$$
	\int_{\omega_1}
	\gamma 
	|u|
	\,
	dV
	-
	\int_{\omega_1}
	\gamma 
	|\bar u|
	\,
	dV
	\le
	C
	\int_\omega
	\gamma^{1 - 1 / p}
	|\nabla u|
	\,
	dV
$$
or, in other words,
\begin{equation}
	\frac{1}{\gamma_1}
	\int_{\omega_1}
	\gamma 
	|u|
	\,
	dV
	\le
	\frac{C}{\gamma_1}
	\int_\omega
	\gamma^{1 - 1 / p}
	|\nabla u|
	\,
	dV
	+
	|\bar u|.
	\label{pl3.2.1}
\end{equation}
By~\eqref{l3.1.1}, we also have
$$
	\int_{\omega_2}
	\gamma 
	|u - \bar u|
	\,
	dV
	\le
	C
	\int_\omega
	\gamma^{1 - 1 / p}
	|\nabla u|
	\,
	dV,
$$
whence in accordance with the inequality $|\bar u| - |u| \le |u - \bar u|$ we immediately obtain
$$
	\int_{\omega_2}
	\gamma 
	|\bar u|
	\,
	dV
	-
	\int_{\omega_2}
	\gamma 
	|u|
	\,
	dV
	\le
	C
	\int_\omega
	\gamma^{1 - 1 / p}
	|\nabla u|
	\,
	dV.
$$
This is obviously equivalent to 
$$
	|\bar u|
	\le
	\frac{C}{\gamma_2}
	\int_\omega
	\gamma^{1 - 1 / p}
	|\nabla u|
	\,
	dV
	+
	\frac{1}{\gamma_2}
	\int_{\omega_2}
	\gamma 
	|u|
	\,
	dV.
$$
Combining the last formula with~\eqref{pl3.2.1}, we complete the proof.
\end{proof}

\begin{lemma}\label{l3.3}
Let the cover~\eqref{3.1} satisfies condition~\eqref{3.2}.
Then
\begin{equation}
	\int_M
	\gamma
	|\varphi|^p
	\,
	dV
	\le
	C
	\int_M
	|\nabla \varphi|^p
	\,
	dV
	\label{l3.3.1}
\end{equation}
for any function $\varphi \in C_0^\infty (M)$, where the constant $C > 0$ 
depend only on $p$, the multiplicity of the cover~\eqref{3.1}, and the left-hand side of~\eqref{3.2}.
\end{lemma}

\begin{proof}
We denote by $C_i > 0$ the constant in~\eqref{l3.1.1} for $\omega = \Omega_i \cup \Omega_{i + 1}$.
Put
$$
	S_i 
	= 
	\sum_{j = 1}^i
	\int_{\Omega_j}
	\gamma
	\,
	dV
	\quad
	\mbox{and}
	\quad
	\gamma_i 
	= 
	\int_{\Omega_i}
	\gamma
	\,
	dV,
	\quad
	i = 1,2,\ldots.
$$ 
Let us also assume that $S_0 = 0$.
We have
\begin{align*}
	&
	\int_M
	\gamma
	|\varphi|^p
	\,
	dV
	\le
	\sum_{i = 1}^\infty
	\int_{\Omega_i}
	\gamma
	|\varphi|^p
	\,
	dV
	=
	\sum_{i = 1}^\infty
	\frac{
		S_i - S_{i - 1}
	}{
		\gamma_i
	}
	\int_{\Omega_i}
	\gamma
	|\varphi|^p
	\,
	dV
	\\
	&
	\qquad
	{}
	=
	\sum_{i = 1}^\infty
	S_i
	\left(
		\frac{1}{\gamma_i}
		\int_{\Omega_i}
		\gamma
		|\varphi|^p
		\,
		dV
		-
		\frac{1}{\gamma_{i + 1}}
		\int_{\Omega_{i + 1}}
		\gamma
		|\varphi|^p
		\,
		dV
	\right),
\end{align*}
whence in accordance with the inequality
\begin{align*}
	\frac{
		1
	}{
		\gamma_i
	}
	\int_{\Omega_i}
	\gamma 
	|\varphi|^p
	\,
	dV
	\le
	{}
	&
	C_i
	\left(
		\frac{
			1
		}{
			\gamma_i
		}
		+
		\frac{
			1
		}{
			\gamma_{i + 1}
		}
	\right)
	p
	\int_{\Omega_i \cup \Omega_{i + 1}}
	\gamma^{1 - 1 / p}
	|\varphi|^{p - 1} |\nabla \varphi|
	\,
	dV
	\\
	&
	{}
	+
	\frac{
		1
	}{
		\gamma_{i + 1}
	}
	\int_{\Omega_{i + 1}}
	\gamma
	|\varphi|^p
	\,
	dV
\end{align*}
which follows from Lemma~\ref{l3.2} we arrive at the estimate
\begin{equation}
	\int_M
	\gamma
	|\varphi|^p
	\,
	dV
	\le
	\sum_{i = 1}^\infty
	S_i
	C_i
	\left(
		\frac{
			1
		}{
			\gamma_i
		}
		+
		\frac{
			1
		}{
			\gamma_{i + 1}
		}
	\right)
	p
	\int_{\Omega_i \cup \Omega_{i + 1}}
	\gamma^{1 - 1 / p}
	|\varphi|^{p - 1} |\nabla \varphi|
	\,
	dV.
	\label{pl3.3.1}
\end{equation}
At the same time, from Jensen's inequality, it follows that
\begin{align*}
	&
	S_i
	C_i
	\left(
		\frac{
			1
		}{
			\gamma_i
		}
		+
		\frac{
			1
		}{
			\gamma_{i + 1}
		}
	\right)
	p
	\int_{\Omega_i \cup \Omega_{i + 1}}
	\gamma^{1 - 1 / p}
	|\varphi|^{p - 1} |\nabla \varphi|
	\,
	dV
	\le
	\varepsilon
	\int_{\Omega_i \cup \Omega_{i + 1}}
	\gamma
	|\varphi|^p
	\,
	dV
	\\
	&
	\qquad
	{}
	+
	A
	S_i^p
	C_i^p
	\left(
		\frac{
			1
		}{
			\gamma_i
		}
		+
		\frac{
			1
		}{
			\gamma_{i + 1}
		}
	\right)^p
	\int_{\Omega_i \cup \Omega_{i + 1}}
	|\nabla \varphi|^p
	\,
	dV
\end{align*}
for any $\varepsilon > 0$, where the constant $A > 0$ depends only on $\varepsilon$ and $p$;
therefore,~\eqref{pl3.3.1} allows us to assert that
$$
	\int_M
	\gamma
	|\varphi|^p
	\,
	dV
	\le
	2
	k
	\varepsilon
	\int_M
	\gamma
	|\varphi|^p
	\,
	dV
	+
	B
	\int_M
	|\nabla \varphi|^p
	\,
	dV
$$
for any $\varepsilon > 0$, where $k$ is the multiplicity of the cover~\eqref{3.1} 
and $B > 0$ is a constant, depending only on $\varepsilon$, $p$, $k$, and the left hand-side of~\eqref{3.2}.
Thus, taking sufficiently small $\varepsilon > 0$ in the last inequality, we complete the proof.
\end{proof}

\begin{lemma}\label{l3.4}
Let the cover~\eqref{3.1} satisfy condition~\eqref{3.3}.
Then inequality~\eqref{l3.3.1} is valid for any function $\varphi \in C^\infty (M)$ equal to zero on $\Omega_1$,
where the constant $C > 0$ depends only on 
$p$, the multiplicity of the cover~\eqref{3.1} and the left-hand side of~\eqref{3.3}.
\end{lemma}

\begin{proof}
We put
$$
	S_i 
	= 
	\sum_{j = i + 1}^\infty
	\int_{\Omega_j}
	\gamma
	\,
	dV
	\quad
	\mbox{and}
	\quad
	\gamma_i 
	= 
	\int_{\Omega_i}
	\gamma
	\,
	dV,
	\quad
	i = 1,2,\ldots.
$$ 
Condition~\eqref{3.3}, in particular, means that $S_i < \infty$ for all positive integers $i$.
As before, by $C_i > 0$ we denote the constant in inequality~\eqref{l3.1.1} for
$\omega = \Omega_i \cup \Omega_{i + 1}$.

It can be seen that
\begin{align*}
	&
	\int_M
	\gamma
	|\varphi|^p
	\,
	dV
	\le
	\sum_{i = 1}^\infty
	\int_{\Omega_{i + 1}}
	\gamma
	|\varphi|^p
	\,
	dV
	=
	\sum_{i = 1}^\infty
	\frac{
		S_i - S_{i + 1}
	}{
		\gamma_{i + 1}
	}
	\int_{\Omega_{i + 1}}
	\gamma
	|\varphi|^p
	\,
	dV
	\\
	&
	\qquad
	{}
	=
	\sum_{i = 1}^\infty
	S_i
	\left(
		\frac{1}{\gamma_{i + 1}}
		\int_{\Omega_{i + 1}}
		\gamma
		|\varphi|^p
		\,
		dV
		-
		\frac{1}{\gamma_i}
		\int_{\Omega_i}
		\gamma
		|\varphi|^p
		\,
		dV
	\right),
\end{align*}
whence in accordance with the inequality
\begin{align*}
	\frac{
		1
	}{
		\gamma_{i + 1}
	}
	\int_{\Omega_{i + 1}}
	\gamma 
	|\varphi|^p
	\,
	dV
	\le
	{}
	&
	C_i
	\left(
		\frac{
			1
		}{
			\gamma_i
		}
		+
		\frac{
			1
		}{
			\gamma_{i + 1}
		}
	\right)
	p
	\int_{\Omega_i \cup \Omega_{i + 1}}
	\gamma^{1 - 1 / p}
	|\varphi|^{p - 1} |\nabla \varphi|
	\,
	dV
	\\
	&
	{}
	+
	\frac{
		1
	}{
		\gamma_i
	}
	\int_{\Omega_i}
	\gamma
	|\varphi|^p
	\,
	dV
\end{align*}
which follows from Lemma~\ref{l3.2} we obtain~\eqref{pl3.3.1}.
In conclusion, it remains to repeat the arguments given in the proof of Lemma~\ref{l3.3}.
\end{proof}

\begin{corollary}\label{c3.1}
If the cover~\eqref{3.1} satisfies condition~\eqref{3.2}, then the manifold $M$ is $p$-hyperbolic.
\end{corollary}

\begin{proof}
Indeed, let $K$ be a compact set of positive measure.
Using Lemma~\ref{l3.3}, we have
$$
	0
	<
	\int_K
	\gamma
	\,
	dV
	\le
	C
	\int_M
	|\nabla \varphi|^p
	\,
	dV
$$
for any function $\varphi \in C_0^\infty (M)$ equal to one in a neighborhood of $K$,
where the constant $C > 0$ does not depend of $\varphi$.
Thus, $\operatorname{cap} (K) > 0$.
\end{proof}

\begin{proof}[Proof of Theorem~$\ref{t3.1}$]
We shall follow the idea given in~\cite{MP}.
Assume that pro\-blem~\eqref{1.1}, \eqref{1.2} has a solution.
In this case, $F$ is a continuous functional in
$\stackrel{\rm \scriptscriptstyle o}{L}\!\!{}_p^1 (M)$.
Let us show the validity of~\eqref{t3.1.1}.
We take functions $\varphi_i \in C_0^\infty (\Omega_i)$ such that
$$
	\| \varphi_i \|_{
		L_p^1 (\Omega_i)
	}
	=
	1
	\quad
	\mbox{and}
	\quad
	(F, \varphi_i)
	\ge
	\frac{1}{2}
	\| F \|_{
		\stackrel{\rm \scriptscriptstyle o}{L}{}_p^1 (\Omega_i)^*
	},
	\quad
	i = 1,2,\ldots.
$$
Putting
$$
	\Phi_j (x)
	=
	\sum_{i=1}^j
	\| F \|_{
		\stackrel{\rm \scriptscriptstyle o}{L}{}_p^1 (\Omega_i)^*
	}^{1 / (p - 1)}
	\varphi_i (x),
	\quad
	j = 1,2,\ldots,
$$
we have
\begin{equation}
	(F, \Phi_j) 
	\ge
	\frac{1}{2}
	\sum_{i=1}^j
	\| F \|_{
		\stackrel{\rm \scriptscriptstyle o}{L}{}_p^1 (\Omega_i)^*
	}^{p / (p - 1)},
	\quad
	j = 1,2,\ldots.
	\label{pt3.1.1}
\end{equation}
On the other hand,
$$
	(F, \Phi_j) 
	\le
	\| F \|_{
		\stackrel{\rm \scriptscriptstyle o}{L}{}_p^1 (M)^*
	}
	\| \Phi_j \|_{
		L_p^1 (M)
	},
	\quad
	j = 1,2,\ldots.
$$
Therefore, taking in to account the relation
$$
	\| \Phi_j \|_{
		L_p^1 (M)
	}^p
	=
	\int_M
	|\nabla \Phi_j|^p
	\,
	dV
	\le
	k^p
	\sum_{i=1}^j
	\| F \|_{
		\stackrel{\rm \scriptscriptstyle o}{L}{}_p^1 (\Omega_i)^*
	}^{p / (p - 1)}
	\int_{\Omega_i}
	|\nabla \varphi_i|^p
	\,
	dV
	=
	k^p
	\sum_{i=1}^j
	\| F \|_{
		\stackrel{\rm \scriptscriptstyle o}{L}{}_p^1 (\Omega_i)^*
	}^{p / (p - 1)},
$$
where $k$ is the multiplicity of the cover~\eqref{3.1}, we obtain
$$
	(F, \Phi_j) 
	\le
	k
	\| F \|_{
		\stackrel{\rm \scriptscriptstyle o}{L}{}_p^1 (M)^*
	}
	\left(
		\sum_{i=1}^j
		\| F \|_{
			\stackrel{\rm \scriptscriptstyle o}{L}{}_p^1 (\Omega_i)^*
		}^{p / (p - 1)}
	\right)^{1 / p},
	\quad
	j = 1,2,\ldots.
$$
Combining the last inequality with~\eqref{pt3.1.1}, we conclude that
$$
	\left(
		\sum_{i=1}^j
		\| F \|_{
			\stackrel{\rm \scriptscriptstyle o}{L}{}_p^1 (\Omega_i)^*
		}^{p / (p - 1)}
	\right)^{1 - 1 / p}
	\le
	2
	k
	\| F \|_{
		\stackrel{\rm \scriptscriptstyle o}{L}{}_p^1 (M)^*
	},
	\quad
	j = 1,2,\ldots,
$$
whence~\eqref{t3.1.1} follows in the limit as $j \to \infty$.

Now, we show that~\eqref{t3.1.1} implies the continuity of the functional $F$ in 
${\stackrel{\rm \scriptscriptstyle o}{L}\!\!{}_p^1 (M)}$.
Really, let $\varphi \in C_0^\infty (M)$. 
Since $\operatorname{supp} \varphi$ is a compact set and~\eqref{3.1} is a locally finite cover,
the support of $\varphi$ can intersect only with a finite number of domains $\Omega_i$.
Consequently,
$$
	\varphi
	=
	\sum_{i=1}^\infty
	\psi_i \varphi,
$$
where almost all terms in the right-hand side are equal to zero, whence we have
\begin{align}
	|(F, \varphi)|
	&
	{}
	\le
	\sum_{i=1}^\infty
	|(F, \psi_i \varphi)|
	\le
	\sum_{i=1}^\infty
	\| F \|_{
		\stackrel{\rm \scriptscriptstyle o}{L}{}_p^1 (\Omega_i)^*
	}
	\| \psi_i \varphi \|_{
		L_p^1 (\Omega_i)
	}
	\nonumber
	\\
	&
	{}
	\le
	\left(
		\sum_{i=1}^\infty
		\| F \|_{
			\stackrel{\rm \scriptscriptstyle o}{L}{}_p^1 (\Omega_i)^*
		}^{p / (p - 1)}
	\right)^{(p - 1) / p}
	\left(
		\sum_{i=1}^\infty
		\| \psi_i \varphi \|_{
			L_p^1 (\Omega_i)
		}^p
	\right)^{1 / p}.
	\label{pt3.1.2}
\end{align}
It is easy to see that
$$
	\| \psi_i \varphi \|_{
		L_p^1 (\Omega_i)
	}^p
	=
	\int_{\Omega_i}
	|\nabla (\psi_i \varphi)|^p
	\,
	dV
	\le
	2^p
	\int_{\Omega_i}
	|\nabla \psi_i|^p |\varphi|^p
	\,
	dV
	+
	2^p
	\int_{\Omega_i}
	\psi_i^p |\nabla \varphi|^p
	\,
	dV,
$$
whence in accordance with~\eqref{3.5} and the fact that $0 \le \psi_i^p \le \psi_i \le 1$ on $\Omega_i$ 
we obtain
$$
	\| \psi_i \varphi \|_{
		L_p^1 (\Omega_i)
	}^p
	\le
	2^p
	\int_{\Omega_i}
	\gamma |\varphi|^p
	\,
	dV
	+
	2^p
	\int_{\Omega_i}
	\psi_i |\nabla \varphi|^p
	\,
	dV,
	\quad
	i = 1,2,\ldots.
$$
Therefore,
$$
	\sum_{i=1}^\infty
	\| \psi_i \varphi \|_{
		L_p^1 (\Omega_i)
	}^p
	\le
	2^p
	k
	\int_M
	\gamma |\varphi|^p
	\,
	dV
	+
	2^p
	\int_M
	|\nabla \varphi|^p
	\,
	dV,
$$
where $k$ is the multiplicity of the cover~\eqref{3.1}. By Lemma~\ref{l3.3}, this implies the estimate
$$
	\sum_{i=1}^\infty
	\| \psi_i \varphi \|_{
		L_p^1 (\Omega_i)
	}^p
	\le
	C
	\int_M
	|\nabla \varphi|^p
	\,
	dV,
$$
where the constant $C > 0$ depends only on $p$, $k$, and the left-hand side of~\eqref{3.2}.
Thus, relation~\eqref{pt3.1.2} allows us to assert that
\begin{equation}
	|(F, \varphi)|
	\le
	C^{1/p}
	\left(
		\sum_{i=1}^\infty
		\| F \|_{
			\stackrel{\rm \scriptscriptstyle o}{L}{}_p^1 (\Omega_i)^*
		}^{p / (p - 1)}
	\right)^{(p - 1) / p}
	\left(
		\int_M
		|\nabla \varphi|^p
		\,
		dV
	\right)^{1 / p}.
	\label{pt3.1.3}
\end{equation}

Theorem~\ref{t3.1} is completely proved.
\end{proof}

\begin{proof}[Proof of Theorem~$\ref{t3.2}$]
The necessity is proved in the same way as in the case of Theorem~\ref{t3.1}.
We only note that since the manifold $M$ is $p$-parabolic, there is a sequence 
$\eta_s \in C_0^\infty (M)$ satisfying~\eqref {t2.2.2}.
This sequence also satisfies~\eqref{t2.2.1} as the existence of a solution of 
problem~\eqref{1.1}, \eqref{1.2} implies that $F$ is a continuous functional in the space
${\stackrel{\rm \scriptscriptstyle o} {L}\!\!{}_p^1 (M)}$.

We prove the sufficiency.
Let $\eta_s \in C_0^\infty (M)$
be a sequence satisfying conditions~\eqref{t2.2.1}, \eqref{t2.2.2}.
Take a Lipschitz domain $\Omega$ with compact closure such that
$K \subset \Omega$ and $\overline{\Omega}_1 \subset \Omega$.
Without loss of generality, it can be assumed that the norms
$
	\| \eta_s \|_{
		W_p^1 (\Omega)
	}
$
are bounded by a constant independent of $s$;
otherwise we replace $\eta_s$ with~\eqref{pt2.2.1}.
Since $W_p^1 (\Omega)$ is completely continuous embedded in $L_p (\Omega)$, 
there exists a subsequence of the sequence $\{ \eta_s \}_{s = 1}^\infty$
converging in ${L_p (\Omega)}$.
For this subsequence we keep the same notation $\{ \eta_s \}_{s = 1}^\infty$.
Taking into account~\eqref{t2.2.2}, one can assert that~\eqref{pt2.2.2} is valid.

We agree to denote by $C$ various positive constants depending only on
$p$, the cover~\eqref{3.1}, the partition of unity $\{ \psi_i \}_{i=1}^\infty$, 
the set $\Omega$, and the left-hand side of~\eqref{3.3}.
Let $\varphi \in {C_0^\infty (M)}$ and, moreover, $\alpha$ be the real number defined by~\eqref{pt2.2.4}.
In view of the Poincare inequality, estimate~\eqref{pt2.2.3} holds.
Also assume that $\varphi_j'$ and $\varphi_j''$ are defined by~\eqref{pt2.2.5},
where $\psi \in C_0^\infty (\Omega)$ is some function equal to one on $\Omega_1$.
For any positive integer $j$ we have
$$
	\varphi_j' 
	=
	\sum_{i=1}^\infty
	\varphi_j'
	\psi_i,
$$
where almost all terms in the right-hand side are equal to zero; therefore,
\begin{align*}
	|(F, \varphi_j')|
	&
	{}
	\le
	\sum_{i=1}^\infty
	|(F, \psi_i \varphi_j')|
	\le
	\sum_{i=1}^\infty
	\| F \|_{
		\stackrel{\rm \scriptscriptstyle o}{L}{}_p^1 (\Omega_i)^*
	}
	\| \psi_i \varphi_j' \|_{
		L_p^1 (\Omega_i)
	}
	\\
	&
	{}
	\le
	\left(
		\sum_{i=1}^\infty
		\| F \|_{
			\stackrel{\rm \scriptscriptstyle o}{L}{}_p^1 (\Omega_i)^*
		}^{p / (p - 1)}
	\right)^{(p - 1) / p}
	\left(
		\sum_{i=1}^\infty
		\| \psi_i \varphi_j' \|_{
			L_p^1 (\Omega_i)
		}^p
	\right)^{1 / p}.
\end{align*}
Thus, replacing in the arguments with which estimate~\eqref{pt3.1.3} was obtained the function $\varphi$ by $\varphi_j '$ and Lemma~\ref{l3.3} by Lemma~\ref{l3.4}, we arrive at the inequality
\begin{equation}
	|(F, \varphi_j')|
	\le
	C
	\left(
		\sum_{i=1}^\infty
		\| F \|_{
			\stackrel{\rm \scriptscriptstyle o}{L}{}_p^1 (\Omega_i)^*
		}^{p / (p - 1)}
	\right)^{(p - 1) / p}
	\left(
		\int_M
		|\nabla \varphi_j'|^p
		\,
		dV
	\right)^{1 / p}.
	\label{pt3.2.2}
\end{equation}
It is not difficult to verify that
$$
	\left(
		\int_M
		|\nabla \varphi_j'|^p
		\,
		dV
	\right)^{1 / p}
	\le
	\| 1 - \psi \|_{
		C (\Omega)
	}
	\| \varphi - \alpha \eta_j \|_{
		L_p^1 (M)
	}
	+
	\|\nabla \psi \|_{
		C (\Omega)
	}
	\| \varphi - \alpha \eta_j \|_{
		L_p (\Omega)
	}
$$
and, moreover,~\eqref{pt2.2.7} and~\eqref{pt2.2.8} are valid; 
therefore,~\eqref{pt3.2.2} implies the estimate
\begin{equation}
	\limsup_{j \to \infty}
	|(F, \varphi_j')|
	\le
	C
	\left(
		\sum_{i=1}^\infty
		\| F \|_{
			\stackrel{\rm \scriptscriptstyle o}{L}{}_p^1 (\Omega_i)^*
		}^{p / (p - 1)}
	\right)^{(p - 1) / p}
	\| \varphi \|_{
		L_p^1 (M)
	}.
	\label{pt3.2.5}
\end{equation}
Since $\operatorname{supp} \psi \subset \Omega$, the function $\varphi_j''$ can be represented as
$$
	\varphi_j''
	=
	\sum_{
		\Omega \cap \Omega_i \ne \emptyset
	}
	(\varphi - \alpha \eta_j)
	\psi
	\psi_i.
$$
We note that the family of domains $\Omega_i$ satisfying the condition $\Omega \cap \Omega_i \ne \emptyset$
is finite as $\overline{\Omega}$ is a compact set and the cover~\eqref{3.1} is locally finite.
Hence, 
$$
	|(F, \varphi_j'')|
	\le
	\sum_{
		\Omega \cap \Omega_i \ne \emptyset
	}
	|(F, (\varphi - \alpha \eta_j) \psi \psi_i)|
	\le
	\sum_{
		\Omega \cap \Omega_i \ne \emptyset
	}
	\| F \|_{
		\stackrel{\rm \scriptscriptstyle o}{L}{}_p^1 (\Omega_i)^*
	}
	\| (\varphi - \alpha \eta_j) \psi \psi_i \|_{
		L_p^1 (\Omega_i)
	}.
$$
At the same time,
$$
	\| (\varphi - \alpha \eta_j) \psi \psi_i \|_{
		L_p^1 (\Omega_i)
	}
	\le
	\| \psi \psi_i \|_{
		C (\Omega)
	}
	\| \varphi - \alpha \eta_j \|_{
		L_p^1 (\Omega)
	}
	+
	\|\nabla (\psi \psi_i) \|_{
		C (\Omega)
	}
	\| \varphi - \alpha \eta_j \|_{
		L_p (\Omega)
	},
$$
whence in accordance with~\eqref{pt2.2.7} and~\eqref{pt2.2.8} we obtain
$$
	\limsup_{j \to \infty}
	\| (\varphi - \alpha \eta_j) \psi \psi_i \|_{
		L_p^1 (\Omega_i)
	}
	\le
	C
	\| \varphi \|_{
		L_p^1 (M)
	}
$$
for all $i$ such that $\Omega \cap \Omega_i \ne \emptyset$.
Thus, one can assert that
$$
	\limsup_{j \to \infty}
	|(F, \varphi_j'')|
	\le
	C
	\sum_{
		\Omega \cap \Omega_i \ne \emptyset
	}
	\| F \|_{
		\stackrel{\rm \scriptscriptstyle o}{L}{}_p^1 (\Omega_i)^*
	}
	\| \varphi \|_{
		L_p^1 (M)
	}.
$$
By the H\"older inequality,
$$
	\sum_{
		\Omega \cap \Omega_i \ne \emptyset
	}
	\| F \|_{
		\stackrel{\rm \scriptscriptstyle o}{L}{}_p^1 (\Omega_i)^*
	}
	\le
	N^{1 / p}
	\left(
		\sum_{
			\Omega \cap \Omega_i \ne \emptyset
		}
		\| F \|_{
			\stackrel{\rm \scriptscriptstyle o}{L}{}_p^1 (\Omega_i)^*
		}^{p / (p - 1)}
	\right)^{(p - 1) / p},
$$
where $N$ is the number of domains $\Omega_i$ satisfying the condition $\Omega \cap \Omega_i \ne \emptyset$;
therefore, 
$$
	\limsup_{j \to \infty}
	|(F, \varphi_j'')|
	\le
	C
	\left(
		\sum_{
			\Omega \cap \Omega_i \ne \emptyset
		}
		\| F \|_{
			\stackrel{\rm \scriptscriptstyle o}{L}{}_p^1 (\Omega_i)^*
		}^{p / (p - 1)}
	\right)^{(p - 1) / p}
	\| \varphi \|_{
		L_p^1 (M)
	}.
$$
Combining this with~\eqref{t2.2.1}, \eqref{pt2.2.6}, and~\eqref{pt3.2.5}, we have
$$
	|(F, \varphi)|
	\le
	C
	\left(
		\sum_{i=1}^\infty
		\| F \|_{
			\stackrel{\rm \scriptscriptstyle o}{L}{}_p^1 (\Omega_i)^*
		}^{p / (p - 1)}
	\right)^{(p - 1) / p}
	\| \varphi \|_{
		L_p^1 (M)
	}.
$$
Theorem~\ref{t3.2} is completely proved.
\end{proof}

\begin{example}\label{e3.1}
Let $M$ be a subset of ${\mathbb R}^n$ of the form
$
	\{ 
		x = (x', x_n)
		: 
		|x'| \le x_n^\lambda,
		\:
		x_n \ge 0
	\}
$
with a smoothed boundary near zero,
where $n \ge 2$ and $\lambda \ge 0$ is some real number.

The manifold $M$ is $p$-hyperbolic if and only if
\begin{equation}
	n > p
	\quad
	\mbox{and}
	\quad
	\lambda > (p - 1) / (n - 1).
	\label{e3.1.1}
\end{equation}
Indeed, if at least one of the inequalities in~\eqref{e3.1.1} is not valid, then taking
$$
	\varphi_{r, R} (x)
	=
	\varphi 
	\left(
		\frac{
			\ln \frac{R}{|x|}
		}{
			\ln \frac{R}{r}
		}
	\right),
	\quad
	0 < r < R,
$$
where $\varphi \in C^\infty ({\mathbb R})$ is some function 
equal to zero in a neighborhood of $(-\infty, 0]$ and to one in the neighborhood of $[1, \infty)$,
we immediately obtain
$$
	\operatorname{cap}_p (\overline{B}_r)
	\le
	\int_M
	\left|
		\nabla \varphi_{r, R}
	\right|^p
	\,
	dV
	\to
	0
	\quad
	\mbox{as } R \to \infty
$$
for all $r > 0$, where $B_r = \{ x \in M : |x| < r \}$.
Thus, $\operatorname{cap} (M) = 0$.

On the other hand, if both inequalities in~\eqref{e3.1.1} are valid, then taking
$
	\Omega_1 
	= 
	\{ 
		x \in M
		:
		|x| < 4
	\},
$
$
	\Omega_i
	= 
	\{ 
		x \in M
		:
		2^{i - 1} < |x| < 2^{i + 1}
	\},
$
$i = 2,3,\ldots$,
and
$\gamma (x) = c (1 + |x|)^{- p}$, where $c > 0$ is enough large real number,
we can construct a partition of unity
$\psi_i \in C_0^\infty (\Omega_i)$
satisfying condition~\eqref{3.5}.
Since~\eqref{3.2} holds, Corollary~\ref{c3.1} implies that $M$ is a $p$-hyperbolic manifold.
\end{example}

\begin{example}\label{e3.2}
Let $M$ be the manifold from Example~\ref{e3.1}.
We assume that $h$ is a measure on $\partial M$ with the density $(1 + |x|)^\sigma$.
If $M$ is a $p$-hyperbolic manifold or, in other words,
inequalities~\eqref{e3.1.1} are fulfilled, then in accordance with Theorem~\ref{t3.1} 
problem~\eqref{c2.1.1}, \eqref{1.2} has a solution if and only if
$$
	\sigma 
	<
	\left\{
		\begin{aligned}
			&
			-
			\frac{\lambda n (p - 1)}{p}
			-
			(1 - \lambda)
			\left(
				2
				-
				\frac{1}{p}
			\right),
			&
			&
			\lambda < 1,
			\\
			&
			- \frac{n (p - 1)}{p},
			&
			&
			1 \le \lambda.
		\end{aligned}
	\right.
$$
Indeed, by estimates based on the embedding theorems, we can show that
$$
	\| h \|_{
		\stackrel{\rm \scriptscriptstyle o}{L}{}_p^1 (\Omega_i)^*
	}
	\asymp
	\left\{
		\begin{aligned}
			&
			2^{
				i
				(
					\sigma
					+
					\lambda n (p - 1) / p
					+
					(1 - \lambda)
					(2 - 1 / p)
				)
			},
			&
			&
			\lambda < 1,
			\\
			&
			2^{
				i
				(
					\sigma
					+
					n (p - 1) / p
				)
			},
			&
			&
			1 \le \lambda,
		\end{aligned}
	\right.
$$
where $\Omega_i$, $i = 1,2,\ldots$, is the cover constructed in Example~\ref{e3.1}.

Let us note that, for $p$-parabolic manifold $M$, problem~\eqref{c2.1.1}, \eqref{1.2} has no solutions
for any $\sigma$ as condition~\eqref{t2.2.1} is not fulfilled.
\end{example}


\end{document}